\documentclass[a4paper, 12pt]{amsart}
\usepackage{fullpage}
\usepackage{graphicx}
\usepackage{tikz}
\usepackage{amsmath,amsthm, amsfonts, amssymb}
\usepackage{mathrsfs}
\usepackage{enumitem}
\usepackage[english]{babel}
\usepackage [autostyle, english = american]{csquotes}
\MakeOuterQuote{"}
\usepackage{tikz-cd}
\usetikzlibrary{cd}
\usepackage{cite}
\usepackage[hidelinks]{hyperref}
\usepackage{amsmath}
\usepackage{xcolor}
\usepackage{verbatim}

\bibliography{bib.bib}
\newcommand{\naturals}{\mathbb{N}}
\newcommand{\complex}{\mathbb{C}}

\newcommand{\A}{\mathcal{A}}
\newcommand{\B}{\mathcal{B}}

\newcommand{\M}{\mathcal{M}}
\newcommand{\N}{\mathcal{N}}

\newcommand{\U}{\mathcal{U}}

\newtheorem{thrm}{Theorem}[section]
\newtheorem{lemma}[thrm]{Lemma}

\newtheorem{cor}[thrm]{Corollary}

\theoremstyle{definition}

\newtheorem{example}[thrm]{Example}
\newtheorem{remark}[thrm]{Remark}

\newtheoremstyle{named}{}{}{\itshape}{}{\bfseries}{.}{.5em}{\thmnote{#3's }#1}
\theoremstyle{named}

\begin{document}
\title{A Class of Freely Complemented von Neumann Subalgebras of $L\mathbb{F}_n$}
\author{Nicholas Boschert, Ethan Davis, AND Patrick Hiatt}

\address{Department of Mathematics, UCLA, Los Angeles, CA 90095}
\email{nickboschert@math.ucla.edu} 

\address{Department of Mathematics\\California Institute of Technology, Pasadena, CA 91125}
\email{erdavis@caltech.edu}

\address{Department of Mathematics, University of California, San Diego, San Diego, CA 92093}
\email{phiatt@ucsd.edu}

\begin{abstract}    We prove that if $A_1, A_2, \dots, A_n$ are tracial abelian von Neumann algebras for $2\leq n \leq \infty$ and $M = A_1 * \cdots * A_n$ is their free product, then any subalgebra $A \subset M$ of the form $A = \sum_{i=1}^n u_i A_i p_i u_i^*$, for some projections $p_i \in A_i$ and unitaries $u_i \in \mathcal{U}(M)$, for $1 \leq i \leq n$, such that $\sum_i u_i p_i u_i^* = 1$, is freely complemented (FC) in $M$. Moreover, if $A_1, A_2, \dots, A_n$ are purely non-separable abelian, and $M = A_1 * \cdots * A_n$,  then any purely non-separable singular MASA in $M$ is FC. We also show that any of the known maximal amenable MASAs $A\subset L\mathbb{F}_n$ (notably the radial MASA), satisfies Popa's \textit{weak FC conjecture}, i.e., there exist  
Haar unitaries $u\in L\mathbb{F}_n$ that are free independent to $A$. 
\end{abstract}

\maketitle

\section{Introduction}

A subalgebra $N$ of a II$_1$ factor $M$ is  \textit{freely complemented}, abbreviated as FC, if there exists a (non-trivial) von Neumann subalgebra 
$Q \subset M$ freely independent to $N$ with $M = N \vee Q$, 
in other words if the inclusion $N\subset M$ is isomorphic to an inclusion of the form $N\subset N*Q$. 
For example, in the free group factor $L\mathbb{F}_n$ generated by freely independent Haar unitaries $u_1, u_2, \dots, u_n$, each of the generator 
maximal abelian $^*$-subalgebras (abbreviated hereafter as {\it MASA}) $A_i := \{u_i\}''$, for $1\leq i\leq n$, is FC. Along these lines, Jekel \cite{J22} proved that some additive perturbations of these generators remain FC. A result in \cite{D97} shows that if $N$ is any diffuse amenable tracial von Neumann algebra and $Q$ is either a free group factor $L\mathbb{F}_n$, for $n \geq 2$, or a diffuse amenable tracial von Neumann algebra, then $N * Q$ is isomorphic to either $L\mathbb{F}_{n+1}$ or $L\mathbb{F}_2$, respectively. Thus, for any $n \geq 2$ and any diffuse amenable tracial von Neumann algebra $N$, there exists an FC embedding of $N$ into $L\mathbb{F}_n$. 

By well known results in \cite{P83}, if $M$ is any II$_1$ factor, then any diffuse amenable subalgebra $B \subset M$ that's freely complemented in $M$ is maximal amenable in $M$, i.e., is contained in no larger amenable von Neumann subalgebra of $M$. In particular, each of the generator MASAs $A_i$ in $L\mathbb{F}_n$ mentioned earlier is maximal amenable. The problem of whether 
conversely any maximal amenable MASA of $L\mathbb{F}_n$ is in fact of this form, and more generally whether any maximal amenable $B\subset L\mathbb{F}_n$ is of the form $B\subset B*Q$, 
which can be viewed as going back to \cite{P83}, has been largely emphasized in recent years by Popa (see Section 5.2 in \cite{P21}, \cite{P23a}, \cite{P23b}; cf. also \cite{BP23}, \cite{BDIP23}) in connection 
with the recent resolution of the Peterson-Thom conjecture in (\cite{H22}, \cite{BeCa}, \cite{BoCo}; see the introduction to \cite{BP23} for a discussion on this).

There are several classes of MASAs in $L\mathbb{F}_n$ that are known to be maximal amenable but for which the FC property is not known (cf. page 6 and 7 in \cite{P23b}): 
$(a)$ The MASA $A_g=\{u_g\}''$  coming from elements $g\in \mathbb{F}_n$ with the cyclic group $g^{\mathbb Z}$ maximal abelian but not freely complemented in $\mathbb{F}_n$, 
which was shown to be maximal amenable in (\cite{P83}, see Remark 3.5.1); 
$(b)$ The radial MASA introduced in \cite{Py81} and shown to be maximal amenable in \cite{CFRW09}; $(c)$ The continuous families of maximal amenable MASAs  in \cite{BP23}.  
Moreover, it was pointed out in (Section 5.5 of \cite{P21}, see also \cite{P23a}, \cite{P23b} and 4.1 in \cite{BDIP23}) that if $\{A_i\}_i$ are disjoint maximal amenable MASAs in a II$_1$ factor 
$M$, then given any projections $p_i \in A_i$ and unitaries $u_i\in M$ with the property that $\sum_i u_ip_iu_i^*=1$, the MASA $A=\sum_i u_i A_ip_i u_i^*$ is maximal amenable in $M$ as well. 
When $M=L\mathbb{F}_n$, this gives an additional class of examples for which the FC property is particularly interesting to investigate. 

As highlighted in \cite{P23b}, the most important, basic such case is when the $A_i$ are all FC, in particular, when $M=A_1* \cdots *A_n$. 
We solve here this latter case proving that, indeed, any MASA ``reassembled'' out of such $A_i$ is freely complemented in $M$. Our result does not in fact depend on the nature of the abelian 
algebras $A_i$, which are allowed to be non-separable as well. 

\begin{thrm}\label{thmA}
    For $2 \leq n \leq \infty$, let $A_1, A_2, \dots, A_n$ be abelian tracial von Neumann algebras, and let $M = A_1 * \cdots *A_n$. Let $p_i \in A_i$, for $1\leq i\leq n$, be projections and $u_i \in \mathcal{U}(M)$, for $1 \leq i \leq n$, be unitaries with $\sum_{i=1}^n u_i p_i u_i^*=1$. Then  $\mathcal{A} =  \sum_{i=1}^n u_i A_ip_i u_i^*$ is freely complemented in $M$.
\end{thrm}

To prove this result, we explicitly construct the  free complement of such a  MASA  by an appropriate  reassembling of the remaining corners $\{A_i(1-p_i)\}_{i=1}^n$. The proof then consists of a 
delicate verification of the free independence of the algebras involved and of the fact that they do indeed generate $M$. 

By combining the above result with results in \cite{BDIP23}, we deduce that if the $A_i$ are assumed purely non-separable $\forall i$, 
then any purely non-separable maximal amenable MASA $A$ in $M=A_1*\cdots *A_n$ is FC. In fact, the mere singularity of $A$ is sufficient to ensure FC, 
with the maximal amenability of $A$ being then automatic, by \cite{P83}. 

\begin{cor}\label{nscorollary}
    Let $A_1, A_2, \dots, A_n$ be purely non-separable abelian tracial von Neumann algebras, $2\leq n \leq \infty$, and $M = A_1 * \dots * A_n$. Let $A\subset M$ be a purely non-separable MASA. 
    Then the following conditions are equivalent. 
    
    \begin{enumerate}
        \item $A$ is singular in $M$.
        \item $A$ is maximal amenable in $M$
        \item $A$ is freely complemented in $M$.         
        \end{enumerate} 
        
\end{cor}

While Popa's general FC problem remains wide open, and it is quite hard to predict what the answer to this problem may be, even in the very concrete cases $(a), (b), (c)$ mentioned above, 
it was conjectured in \cite{P23a}, \cite{P23b} that the following weaker property  
holds  true for all MASAs in free group factors: {\it Popa's weak FC conjecture} ``Given any amenable von Neumann subalgebra $B$ of a free group factor $L\mathbb{F}_n$ (like for instance a MASA), 
there exists a Haar unitary $u\in L\mathbb{F}_n$ that's free independent to $B$.'' 

We approach this conjecture in the second part of the paper, and prove that it is indeed verified by  
each one  of the  above  classes $(a), (b), (c)$ 
of examples of maximal amenable MASAs in $L\mathbb{F}_n$. This is in fact easy to see for the examples $(a)$, 
because any $g\in \mathbb{F}_n$ admits plenty of $h\in \mathbb{F}_n$ that are free to $g$, and and also for the examples $(b)$, where 
each one of the maximal amenable MASAs $A(t)$ constructed in \cite{BP23} has $A(t')$ free independent to $A(t)$ whenever $t'\perp t$ in the $\mathbb R$-valued $\ell^2I$. In the case of  
the radial MASA $A_r\subset L\mathbb{F}_n$, $2\leq n < \infty$, in example $(c)$, generated by the element $r=\sum_{i=1}^n u_i+u_i^*$, where $u_1, \dots, u_n$ are the free generators of 
$L\mathbb{F}_n$, we use the fact that $B_i:=\{u_i+u_i^*\}''\subset \{u_i\}''=:A_i$ corresponds to the inclusion of the algebra of even functions into $L^\infty[-1,1]$. Thus, the 
unitary $u_i=|x|/x\in L^\infty[-1,1]$ in $A_i$ has period 2 and $0$ expectation onto $B_i$, and we show the following: 

\begin{thrm}\label{thmB} Any of the  above unitaries $v_{ij}=u_iu_j \in L\mathbb{F}_n$, $1\leq i\neq j \leq n$,  is free independent to 
the radial MASA  $A_r \subset L\mathbb{F}_n$, $2\leq n <\infty$.    
\end{thrm}

We in fact treat this example within a more general framework, obtaining many more examples where the weak FC property is verified. For instance we show that if $A_i$ are all 
purely non-separable then any abelian von Neumann subalgebra in $A_1* \cdots *A_n$ admits Haar unitaries that are free independent to it (see Section 3).

\textbf{Acknowledgements:}
We are much indebted  to Adrian Ioana for suggesting the method of proof of Lemma \ref{Existence of orthogonal unitary lemma}. 
We are very grateful to Sorin Popa for strongly encouraging  us to work on these problems 
and to Dima Shlyakhtenko for many illuminating conversations regarding this work.

\section{Free Reassembly Algebras}
Throughout this section, for $2 \leq n \leq \infty$, let $A_1, A_2, \dots, A_n$ be diffuse abelian tracial von Neumann algebras, and let $M = A_1 * \cdots *A_n$ denote their tracial free product (c.f. \cite{C73, V85}). Following \cite{P21}, we consider the following class of maximal amenable subalgebras $\mathcal{A} \subset M$. 
For $1\leq i \leq n$, choose projections $p_i \in A_i$ and unitaries $u_i \in \mathcal{U}(M)$ such that $\sum_{i=1}^n u_i p_i u_i^*=1$. Then let $\mathcal{A} = \sum_{i=1}^n u_i A_ip_i u_i^*$

Such an abelian subalgebra $\mathcal{A} \subset M$ can be thought of as a "free reassembly" of the generating algebras $A_1, \dots, A_n$ of $M$. By a result of Popa in \cite{P83}, one immediately sees that the MASA $\A$ is maximal amenable in $M$.  In the case the algebras $A_1, \dots, A_n$ are separable and diffuse, so that $M = L\mathbb{F}_n$, this class of maximal amenable subalgebra was put forward in \cite{P21} as a "test candidate" for the FC Problem (see also \cite{P23a}, \cite{P23b}, \cite{BDIP23}).

We first observe that up to unitary conjugacy, $\mathcal{A}$ is independent of the choice of unitaries $u_i$, for $1\leq i \leq n$.

\begin{lemma} \label{Lemma: Independence of Choice of Unitaries}

    For $1\leq i \leq n$, let $p_i \in A_i$ be projections and $u_i, v_i \in \mathcal{U}(M)$ be unitaries with $\sum_{i=1}^n u_ip_iu_i^* = \sum_{i=1}^n v_ip_iv_i^* = 1$. Then the algebras $\mathcal{A}_1 = \sum_{i=1}^n u_iA_ip_iu_i^*$ and $\mathcal{A}_2 = \sum_{i=1}^n v_iA_ip_iv_i^*$ are unitarily conjugate in $M$.

\end{lemma}
\begin{proof}
Let $w = \sum_{i=1}^n v_ip_iu_i^*$. Since the projections $u_ip_iu_i^*$, for $1\leq i\leq n$, are pairwise orthogonal, it is clear that $ww^* = \sum_{i=1}^n v_ip_iv_i^* = 1$. Hence, $w$ is a unitary. It is then straightforward to see that $w\mathcal{A}_1w^* = \mathcal{A}_2$. 
\end{proof}

Given a choice of MASA $\mathcal{A} = \sum_{i=1}^n u_i A_ip_i u_i^*$ as described above, we would like to use Lemma \ref{Lemma: Independence of Choice of Unitaries} to modify the unitaries $u_i$ to a specific form. To do this, we require a small technical lemma regarding the structure of $W^{*}(p,q)$ when $p$ and $q$ are freely independent projections that have the same trace.

\begin{lemma}\label{Existence of unitary}
    Let $M$ be a tracial von Neumann algebra and let $p_1, p_2 \in M$ be freely independent projections with $\tau(p_1) = \tau(p_2)$. Then there is a unitary $u \in \{p_1, p_2\}''$ such that $up_1u^* = p_2$. Moreover, one can choose this unitary $u$ to be self adjoint.
\end{lemma}
\begin{proof}
    Without loss of generality, we may assume that $\tau(p_1) = \tau(p_2) \leq 1/2$, since otherwise we can apply the lemma to $1-p_1$ and $1-p_2$. Consider then the self adjoint elements $x = p_1 + p_2 - 1$ and $y = p_1 - p_2$. One can see that 
    $$xy = -p_1p_2 + p_2p_1 = -yx$$
    so that $x$ and $y$ anticommute. If we take $x=u|x|$ to be the polar decomposition of $x$, then $u$ is an odd Borel measurable function of $x$, so $ux = xu$ and $uy = - yu$. Adding these two equations together and simplifying then gives $up_1 = p_2u$. 

    It suffices then to check that $u$ is unitary, i.e. that it has no kernel, or that $0$ is not in its point spectrum. In turn, this means it is enough to know that 0 is not in the point spectrum of $x$. For this, we can use Voiculescu's $R$-transform (c.f. \cite{V86}) to consider the spectral measure of $p_{1}+p_{2}$:

    The Cauchy transform of each projection is 
    \[G_{p}(z)=\frac{\tau(p)}{1-z}+\frac{(1-\tau(p))}{z}\]
    which produces the $R$-transform 
    \[R(w)=\frac{w-1-2\tau(p)-\sqrt{(1-2\tau(p))^{2}-w(2-w)}}{2w}\]
    Doubling this, adding $\tfrac{1}{w}$, and inverting yields the Cauchy transform of $p_{1}+p_{2}$:
    \[G_{p_{1}+p_{2}}(z)=\frac{-2\tau(p)(z-1)-1+\sqrt{(2\tau(p)-1)^{2}(z-1)^{2}-4\tau(p)(z-1)+4\tau(p)}}{(z-1)^{2}-1}\]
    From this it can be easily seen that the only atoms in the spectral measure of $p_{1}+p_{2}$ can be at $z=2$ or $z=0$, since $(z-a)G_{p_{1}+p_{2}}(z)$ goes to zero as $z$ goes to $a$ for every other $a$ (in particular, for $1$).

    Finally, note that this $u$ is indeed self adjoint, since it is an $\mathbb{R}$ valued Borel function of a self adjoint operator.
\end{proof}

Before we begin the proof of the main result, we require one more small technical lemma.
\begin{lemma}\label{lemma: technical 2-norm convergence lemma}
Let $M$ be a tracial von Neuamann algebra, and let $m_1, m_2, \dots, m_k \in M$. Assume that for some bounded operators $\xi_{i, n} \in M$, for $1 \leq i \leq k$ and $n \in \naturals$, that we can express $m_i$ as a $\| \ \|_2$-norm convergent sum $m_i = \sum_{n = 1}^\infty \xi_{i, n}$. Then if for all $k$-tuples $(n_1, \cdots, n_k)$ of positive integers we have $\tau(\prod_{i=1}^k \xi_{i, n_i}) = 0$ then $\tau(\prod_{i=1}^k m_i) = 0$.
\end{lemma}
\begin{proof}
    We claim that for any $1\leq l \leq k$ and any choice of positive integers $n_{l+1}, \dots, n_k$ that $\tau(\prod_{i=1}^l m_i \prod_{i = l+1}^k\xi_{i, n_{i}} ) = 0$.
    The claim is easy to see for $l = 1$. Indeed, fix a choice of positive integers $n_2, \dots, n_k$. By assumption we can express $m_1$ as a $\| \ \|_2$-norm convergent series $m_1 = \sum_{n = 1}^\infty \xi_{1, n}$. But since the operator $\prod_{i = 2}^k \xi_{i, n_i} \in M$ is bounded we get the series 
    $$m_1 \prod_{i = 2}^k \xi_{i, n_i} = \sum_{n = 1}^\infty \xi_{1, n}\xi_{2, n_2} \cdots \xi_{k, n_k}$$
    is also $\| \ \|_2$-norm convergent. By assumption, each summand $\xi_{1, n}\xi_{2, n_2} \cdots \xi_{k, n_k}$, for $n \geq 1$, has trace 0, hence $\tau(m_1 \prod_{i = 2}^k \xi_{i, n_i}) = 0$ as well. Proceeding with induction on $l$, the claim quickly follows. The Lemma then follows by taking $l = k$.
\end{proof}
\begin{remark}\label{remark: formal series expansion}
    Note we can interpret the above Lemma in the following manner. Let $M$ be a tracial von Neumann algebra and $m_1, m_2, \cdots, m_k \in M$. Suppose we have $\| \ \|_2$-norm convergent series representations $m_i = \sum_{n = 1}^\infty \xi_{i, n}$ for each $1 \leq i \leq k$ for bounded operators $\xi_{i, n}\in M$. Then imagine we expand the product $m_1m_2\cdots m_k$ as a formal series 
    $$\sum_{1 \leq n_1, \dots, n_k} \xi_{1, n_1}\cdots \xi_{k, n_k}$$
    In general, this series might not converge in any meaningful way. However, as long as we can show that each summand $\xi_{1, n_1}\cdots \xi_{k, n_k}$ has trace 0, we can still conclude $\tau(m_1\cdots m_k) = 0$ as well.
\end{remark}

\begin{thrm}\label{thm: free reassemblies are freely complemented}
    For $2 \leq n \leq \infty$, let $A_1, A_2, \dots, A_n$ be abelian tracial von Neumann algebras, and let $M = A_1 * \cdots *A_n$. Let $p_i \in A_i$, for $1\leq i\leq n$, be projections and $u_i \in \mathcal{U}(M)$, for $1 \leq i \leq n$, be unitaries with $\sum_{i=1}^n u_i p_i u_i^*=1$. Then  $\mathcal{A} =  \sum_{i=1}^n u_i A_ip_i u_i^*$ is freely complemented in $M$.
\end{thrm}
\begin{proof}
    We begin by examining the case when $n = 2$. Let $A_1$ and $A_2$ be abelian tracial von Neumann algebras. As in the problem statement, let $p_1 \in A_1$ and $p_2 \in A_2$ be projections, and let $u_1, u_2 \in \U(M)$ be unitaries such that $\sum_{i = 1}^2 u_ip_iu_i^* = 1$. Let $\A = \sum_{i=1}^2 u_i A_i p_i u_i^*$. Then we show that $\A$ is freely complemented in $M = A_1 * A_2$. 

    For sake of notation, call $p_{11} = p_1$, $p_{12} = 1 - p_1$, $p_{21} = p_2$, and $p_{22} = 1 - p_2$. By Lemma \ref{Existence of unitary}, there is a unitary $u \in \{p_{12}, p_{21}\}''$ such that $up_{12}u^* = p_{21}$. Since this unitary can be taken to be self-adjoint, we actually have $up_{12}u = p_{21}$ and $up_{11}u = p_{22}$. Consider then the two algebras $\A_1 = A_1 p_{11} + uA_2p_{21}u$ and $\A_2 = uA_1p_{12}u + A_2 p_{22}$. Notice by Lemma \ref{Lemma: Independence of Choice of Unitaries} that $\A_1$ is unitarily conjugate to the original algebra $\A$. It suffices then to check that $\A_1$ and $\A_2$ are freely independent and generate all of $M$.

    That $\A_1$ and $\A_2$ generate all of $M$ is a direct consequence of the choice of unitary $u$. Indeed, $\{\A_1, \A_2\}''$ contains each of the corner algebras $A_1 p_{11}$, $uA_2p_{21}u$, $uA_1p_{12}u$, and $A_2 p_{22}$. In particular, it also contains each projection $p_{ij}$, for $1\leq i, j \leq 2$ and hence by the construction in Lemma \ref{Existence of unitary}, it contains the unitary $u$. After conjugation, it follows immediately then that $\{\A_1, \A_2\}''$ contains each of the corner algebras $A_jp_{ji}$, for $1 \leq i, j \leq 2$, hence $M = \{\A_1, \A_2\}''$.

    It remains to check the algebras $\A_1$ and $\A_2$ are freely independent. To do this, we begin by fixing some notation and examining the algebra $\{p_{12}, p_{21}\}''$ from which we chose the unitary $u$. For $i = 1, 2$, call $$v_{i} = (p_{ii} - \alpha) / \sqrt{\alpha - \alpha^2}$$ where $ \alpha = \tau(p_{11}) = \tau(p_{22})$. Notice that $v_1$ and $v_2$  both have trace 0 and are freely independent to each other, and, moreover,  with this normalization satisfy $v_{i}^2 = 1 + cv_{i}$, for $i = 1, 2$, where $c$ is the constant $(1 - 2\alpha) / \sqrt{\alpha - \alpha^2}$. 
    
    Consider then the set $S$ of all alternating words of the form $v_{i_1}v_{i_2} \cdots v_{i_k}$, where $i_l \neq i_{l+1}$ for $1\leq l \leq k$. Here if $k=0$, we allow the word to be just $1$. Then based on the previously mentioned observations and freeness of $A_{1}$ and $A_{2}$,  we see that $S$ forms an orthonormal basis for the $L^2$-space of $\{p_{12}, p_{21}\}''$. In particular, we can write the unitary $u$ as a $\| \ \|_2$-convergent sum $u = \sum_{\xi \in S} a_{\xi} \xi$ for some constants $a_\xi \in \complex$. 

    We return now to check that $\A_1$ and $\A_2$ are freely independent. Consider an alternating word $X = a_1 a_2 \cdots a_k$ of trace 0 terms $a_l \in \A_{i_l}$, for $1 \leq l \leq k$, with $i_l \neq i_{l+1}$ for $1 \leq l \leq k - 1$. Then we check $\tau(X) = 0$.

    Notice by linearity that any term $x$ with trace 0 in $\A_1$ can be decomposed as a sum $x = x_1 p_{11} + ux_2p_{21}u + \lambda v_1$ for some operators $x_1 \in A_1p_{11}$ and $x_2 \in A_2p_{21}$ with trace 0 and some constant $\lambda \in \complex$. Similarly, any term $y$ with trace 0 in $\A_2$ can be decomposed as a sum $y = u y_1 p_{12}u + y_2p_{22} + \gamma v_2$ for some operators $y_1 \in A_1p_{12}$ and $y_2 \in A_2p_{22}$ with trace 0 and some constant $\gamma \in \complex$. So without loss of generality, we may assume that each of the terms $a_l$, for $1 \leq l \leq k$, in the alternating word $X$ are either $v_i$, for $i = 1, 2$, or are trace 0 terms from one of the corner subalgebras $A_1 p_{11}$, $uA_2p_{21}u$, $uA_1p_{12}u$, or $A_2 p_{22}$.

    After this assumption regarding the form of the $a_l$, we claim, in the spirit of Remark \ref{remark: formal series expansion}, that $X$ can be written as a formal sum $X = \sum_{i = 1}^\infty w_i$, where each $w_i$ is an alternating word of trace 0 terms from the generator MASAs $A_1$ and $A_2$ of $M$. This sum is generated by applying a sequence of algebraic reductions, replacing each occurrence of $u$ with its series expansion in terms of the elements of $S$, and then applying additional algebraic simplifications. From here, it will follow by Lemma \ref{lemma: technical 2-norm convergence lemma} that $\tau(X) = 0$, since each of the alternating words $w_i$ will have trace 0.
    
    We present this sequence of algebraic simplifications of $X$ below notating them as steps (1) through (4).

    $(1)$ First, after expanding the word $X$ imagine that we have a subword of the form $u v_{i_1}v_{i_2} \cdots v_{i_r}u$, for some $r \leq k$, with $i_l \neq i_{l+1}$ for all $1\leq l \leq r-1$. Then since $up_{11}u = p_{22}$ and $up_{22}u = p_{11}$ by construction, we get that $uv_1u = v_2$ and $uv_2u = v_1$. But then 
    $$u v_{i_1}v_{i_2} \cdots v_{i_r}u
    =
    \prod_{j = 1}^r uv_{i_j}u
    =
    \prod_{j = 1}^r v_{3 - i_j}
    $$
    Thus, we can replace this substring with the product $v_{3- i_1}v_{3- i_2} \cdots v_{3 - i_r}$. That is, we swap each of the $v_1$ with $v_2$ and vice versa, and the $u$ terms disappear. 

    $(2)$ Next, look at the remaining locations in the word $X$ where a $u$ occurs. As we saw before, we can represent the unitary $u$ as a $\| \ \|_2$-norm convergent sum $u = \sum_{\xi \in S} a_{\xi} \xi$ for some constants $a_\xi \in \complex$. Substitute in this series expansion for each remaining term $u$ in the word $X$. We can expand the product $X$ by distributing over each of these sums to write $X$ as a formal sum $X = \sum_{i\geq 1} w_i^{(0)}$, where each $w_i^{(0)}$ is a product of elements $\xi \in S$ and trace 0 terms from the corner subalgebras $A_i p_{ij}$ for $1 \leq i, j \leq 2$. Although this formal series might not converge in any meaningful way, by Lemma \ref{lemma: technical 2-norm convergence lemma} we have that $\tau(X) = 0$ as long as each of these words $w_i^{(0)}$ has trace 0.

    $(3)$ Now we examine each of these words $w_i^{(0)}$. For an index $i$, consider if   somewhere in the product $w_i^{(0)}$ we have a segment $\xi_1 \xi_2$ with $\xi_1, \xi_2 \in S$. However, using the formulas $v_i^2 = 1 + cv_i$ we mentioned earlier, it is clear $\xi_1\xi_2$ is a finite linear combination of elements in $S$. Now replace $\xi_1 \xi_2$ with such a linear combination in the word $w_i^{(0)}$ and distribute this product out over the addition. Repeating this procedure, we can further rewrite $X$ as a formal sum $X = \sum_{i\geq 1} w_i^{(1)}$, where each $w_i^{(1)}$ is a product of elements $\xi \in S$ and trace 0 terms from the corner subalgebras $A_i p_{ij}$ for $1 \leq i, j \leq 2$, such that no $w_i^{(1)}$ contains adjacent terms $\xi_1, \xi_2 \in S$. Again, this series may not converge in any meaningful way, but note that we will still have $\tau(X) = 0$ as long as each of the words $w_i^{(1)}$ has trace 0.

    $(4)$ Now examine each of these new words $w_i^{(1)}$. View each $w_i^{(1)}$ as a product of the terms $v_1$ and $v_2$ and trace 0 terms $x_{i j}p_{i j} \in A_i p_{i j}$ for $1 \leq i,j \leq 2$. Imagine that for some index $l$ that the product $w_l^{(1)}$ contains a segment $x_{i j}p_{i j} v_i$. However, $p_{i j} v_i$ will be a scalar multiple of $p_i$, so $x_{i j}p_{i j} v_i$ will still be a trace 0 term in $A_ip_{i j}$. Hence, we can replace $x_{i j}p_{i j}v_i$ with $x_{i j}'p_{i j}$ for some trace 0 term $x_{i j}'$ in $A_ip_{i  j}$. By symmetry we can make a similar change if the product $w_l^{(1)}$ contains a segment $v_ix_{i j}p_{i j}$. After making these changes, call $w_l$ the resulting words. 

    After steps $(1)$-$(4)$, we have expressed $X$ as a formal sum $X = \sum_{i = 1}^\infty w_i$ in a way such that $\tau(X) = 0$ as long as each word $w_i$ has tace 0. We claim that in fact each $w_i$ is an alternating word of trace 0 terms from the generator MASAs $A_1$ and $A_2$. Indeed, by construction, each $w_l$ is a product of the terms $v_1$ and $v_2$ and trace 0 terms $xp_{i j} \in A_i p_{i  j}$. By steps $(2)$ and $(3)$ of the reduction, any two adjacent terms $v_{i_1}$ and $v_{i_2}$ in $w_{l}$ must lie in different generator MASAs. More over by step $(4)$, any two adjacent terms $xp_{i_1,j_1}$ and $v_{i_2}$ lie in different generator MASAs. Finally, notice that any two adjacent terms $xp_{i_1, j_1}$ and $yp_{i_2, j_2}$ have $(i_1, j_2) \neq (i_2, j_2)$. Indeed the only way after applying the reductions in steps $(2)$-$(4)$ we could have a segment $xp_{ij}yp_{ij}$ in $w_l$ would be if $X$ orinigally contained a segment $xp_{ij}u v_{i_1}v_{i_2} \cdots v_{i_r}uyp_{ij}$, for some $r \leq k$, with $i_l \neq i_{l+1}$ for all $1\leq l \leq r-1$. However, the reduction in step $(1)$ would have already replaced this segment with an alternating product of trace 0 terms from the generator MASAs $A_1$ and $A_2$. Thus, that any two adjacent terms $xp_{i_1, j_1}$ and $yp_{i_2, j_2}$  in $w_l$ must have $(i_1, j_2) \neq (i_2, j_2)$. But then this pair of adjacent terms $xp_{i_1, j_1}$ and $yp_{i_2, j_2}$ would either lie in different generator MASAs or be supported under orthogonal projections, and so have product 0.
    
    It follows then that $w_l$ is either equal to 0 or is an alternating word of trace 0 terms from the generator MASAs $A_1$ and $A_2$, and hence in either case $\tau(w_l) = 0$. By Lemma \ref{lemma: technical 2-norm convergence lemma} we will have $\tau(X) = 0$ as desired, so the result for $n = 2$ follows. 

    For $2 \leq n \leq \infty$ arbitrary, we use a similar argument. As in the problem statement, let $A_1, A_2, \dots, A_n$ be abelian tracial von  Neumann algebras and let $M = A_1 * \cdots * A_n$. Choose projections $p_i \in A_i$ and unitaries $u_i$, for $1\leq i\leq n$, such that $\sum_{i=1}^n u_i p_i u_i^* = 1$. Then we show the MASA $\A = \sum_{i=1}^n u_i A_ip_iu_i^*$ is FC. 

    Indeed, fix a partition of unity $\sum_{i=1}^n q_i = 1$ in $A_1$ such that $\tau(q_i) = \tau(p_i)$, for $1\leq i \leq n$. For $2 \leq i \leq n$, let $u_i \in \{p_i, q_i\}''$ be the self adjoint unitary from Lemma \ref{Existence of unitary} such that $u_i p_i u_i^* = q_i$. Then let $\A_1 = A_1q_1 + \sum_{i=2}^n u_iA_ip_iu_i^*$ and $\A_i = A_i(1-p_i) + u_iA_1q_iu_i^*$, for $2\leq i \leq n$. By Lemma \ref{Lemma: Independence of Choice of Unitaries}, $\A_1$ will be unitarily conjugate to $\A$. Moreover, after inductively repeating the argument for $n=2$, we see immediately that $\A_1, \dots, \A_n$ are freely independent and generate $M$. Hence, $\A$ is FC.
\end{proof}

Note that in the case when $A_1, A_2, \dots, A_n$ are separable diffuse abelian algebras, the above result shows that any of the maximal amenable MASA's in $L(\mathbb{F}_n)$ 
that are obtained by free reassembling of the $A_i$'s, highlighted in \cite{P21}, \cite{P23b}, are indeed FC. 

As an immediate consequence of Theorem \ref{thm: free reassemblies are freely complemented} and of a result in \cite{BDIP23}, we obtain the following positive solution to a ``non-separable version'' of Popa's FC problem: 

\begin{cor}\label{non-separable FC Theorem}
    Let $A_1, A_2, \dots, A_n$ be purely non-separable abelian tracial von Neumann algebras, $2\leq n \leq \infty$, and $M = A_1 * \dots * A_n$. Let $A\subset M$ be a purely non-separable MASA. 
    Then the following conditions are equivalent. 
    \begin{enumerate}
        \item $A$ is singular in $M$.
        \item $A$ is maximal amenable in $M$
        \item $A$ is freely complemented in $M$.         
        \end{enumerate} 
\end{cor}
\begin{proof}
    That $(3) \Rightarrow (2) \Rightarrow (1)$ is clear. Conversely to see $(1) \Rightarrow (3)$, let $A \subset M$ be a purely non-separable singular MASA. By Corollary 3.7 of \cite{BDIP23}, there exist projections $p_i \in A_i$ and unitaries $u_i \in \U(M)$, for $1\leq i \leq n$, such that $\sum_{i=1}^n u_i p_i u_i^* = 1$ and $\A = \sum_{i=1}^n u_i A_ip_i u_i^*$. The result then follows directly from Theorem \ref{thm: free reassemblies are freely complemented}. 
\end{proof}

\begin{remark}
Let $M=A^{*n}$ with $A$ a purely non-separable abelian tracial von Neumann algebra. As noted in \cite{BDIP23}, Section 4.2, our positive answer to the FC Problem for the free reassembly algebras furnishes a partial description of the outer automorphism group ${\rm Out}(M):={\rm Aut}(M)/{\rm Int}(M)$ by relating it with automorphisms of the sans-core, defined as follows in \cite{BDIP23}. The \textit{sans-core} is the unique (up to unitary conjugacy) maximal abelian purely non-separable $*$-subalgebra $\A\subset \M := B(\ell^2 K) \overline{\otimes} M $ (for some $|K|\geq 2^{|\U(M)|}$) generated by finite projections with $\A\subset 1_{\A} \M 1_{\A}$ singular \cite{BDIP23}. As suggested in \cite{BDIP23}, we obtain that there is an injective group morphism from the group $G$ of $\operatorname{Tr}$-preserving automorphisms of the unfolded form of the sans-core $\mathcal{A}^{\rm ns}_M$ (c.f. \cite{BDIP23}) to $\operatorname{Out}(M)$.

Indeed,  let $\theta\in \operatorname{Aut} (M)$. Note $\theta$ is automatically trace-preserving since $M$ is a factor. Now, $\theta$ sends each generator copy of $A$ to another purely non-separable singular MASA in $M$, and moreover, by freeness, is completely determined by where it sends each of these copies. But by Theorem \ref{non-separable FC Theorem}, every purely non-separable singular MASA in $M$ is freely complemented. So,  automorphisms $\theta$ of $M$ are in bijection with free decompositions of $M$ into $n$ free reassembly algebras. By definition of the sans-core, any $\rm{Tr}$-preserving automorphism of $\A^{\rm ns}_M$ determines a unique (up to unitary conjugation) such decomposition of $M$, and the mapping thus obtained is easily seen to be a homomorphism.

\end{remark}
In summary, we have the following corollary.

\begin{cor}
    Let $M=A^{*n}$ with $A$ a purely non-separable abelian tracial von Neumann algebra. Let $\mathcal{A}^{\rm ns}_M$ denote the unfolded form of the sans-core of $M$ (c.f. \cite{BDIP23}). Then there is an injective group morphism $\Phi:\rm{Aut}(\mathcal{A}^{\rm ns}_M, \rm{Tr}) \to \rm{Out}(M)$.
\end{cor}

\section{The weak FC property for some concrete MASAs}

In this section we discuss Popa's weak FC conjecture mentioned in the introduction, which predicts that any maximal amenable von Neumann subalgebra $B$ of a free group factor $M$ 
admits Haar unitaries $u\in M$ that are free independent to $B$. We in fact concentrate on the case $B$ is abelian, 
 discussing each one of the concrete cases of maximal amenable MASAs that are known, mentioned already in the introduction, and which we remind here for convenience. We will show that all these examples satisfy the conjecture.

\begin{example} (cf. \cite{P83}, \cite{P23b}) 
    Consider the MASA $A_g = \{u_g\}''$ of $L\mathbb{F}_n$ coming from an element $g \in \mathbb{F}_n$ with the cyclic group $g^{\mathbb{Z}}$ maximal abelian. The MASA $A_g$ was shown to be maximal amenable in \cite{P83}. However, in the case when the cyclic group $g^{\mathbb{Z}}$ is not freely complemented as a subgroup of $\mathbb{F}_n$ it remains open whether $A_g$ is FC. 
    For instance, if $a, b\in  F_2$ are the generators and $g=aba^{-1}b^{-1}$, then it is not known whether $A_g=\{u_g\}''$ is FC or not. 
    Note however that if $h \in \mathbb{F}_n$ is any word freely independent to $g$ then $A_g$ and $A_h$ will be freely independent in $L\mathbb{F}_n$. 
    The construction of such a freely independent word $h$ is immediate, thus each of these algebras $A_g$ satisfies the weak FC property. 
\end{example}

\begin{example} (cf. \cite{BP23}) 
    View $L\mathbb{F}_n = M_1 * \cdots * M_n$ for diffuse amenable subalgebras $M_i$, $1 \leq i \leq n$, and for each $i$ fix a semicircular element $s_i \in M_i$ 
    (not necessarily generating $A_i$!). For any tuple of real numbers $t \in \mathbb{R}^n$, call $s(t) = \sum_{i = 1}^n t_i s_i$ and let $A(t) = \{s(t) \}'' \subset L\mathbb{F}_n$. Then as shown in \cite{BP23}, each of the MASAs $A(t) \subset M$, with $t$ having at least two non-zero entries,  is maximal amenable in $L\mathbb F_n$. While it is not clear if these $A(t)$ are FC, note that if $t$ and $t'$ are perpendicular vectors in $\mathbb{R}^n$ then $A(t)$ is free independent to $A(t')$. Indeed, let $N \subset M$ be the von Neumann subalgebra generated by the semicircular elements $s_i$, for $1 \leq i \leq n$. Then we may view $N$ as a subalgebra of bounded operators on free Fock Space $\mathcal{F}(\mathbb{R}^n)$ (c.f. \cite{V85}). Here $s(t) = \ell(t) + \ell(t)^*$, where $\ell(t)$ is the canonical creation operator on $\mathcal{F}(\mathbb{R}^n)$. A well known theorem states if $t \perp t'$ then $\ell(t) + \ell(t')^*$ and $\ell(t) + \ell(t')^*$ will be freely independent, and hence $A(t)$ will be freely independent to $A(t')$. 
\end{example}

Another example of a MASA of the free group factor $L \mathbb{F}_n$  that we have mentioned in the introduction is the so called radial subalgebra $A_r$. That is the algebra generated by the element $\sum_{i = 1}^n u_i + u_i^*$ where $u_i$, for $1 \leq i \leq n$, are the freely independent Haar unitaries that generate $L \mathbb{F}_n$. The radial MASA was introduced in \cite{Py81} and shown to be maximal amenable in \cite{CFRW09}. Unlike the previous two classes of MASAs, it is not clear a priori that the radial subalgebra $A_r$ satisfies the weak FC conjecture. To that end, we prove the following lemma.

\begin{lemma}\label{weak FC property lemma}
    Let $A_1$ and $A_2$ be abelian tracial von Neumann algebras, and let $M = A_1 * A_2$ be their free product. Let $B_1 \subset A_1$, $B_2 \subset A_2$ be von Neumann subalgebras. Suppose $u_1 \in A_1$ and $u_2 \in A_2$ are unitaries such that $E_{B_i}^{A_i}(u_i) = 0$, for $i = 1, 2$. Then $u = u_1u_2$ is freely independent to $B_1 * B_2$.
\end{lemma}

\begin{proof}
    Let $u = u_1u_2$ and let $A = \{u\}''$.  Note that $A$ is a diffuse abelian subalgebra of $M$ with the set of all powers $(u_1u_2)^i$ and $(u_2^*u_1^*)^i$, for $i\geq 1$, forming an orthogonal basis for $L^2A \ominus \complex 1$.

    To see that $A$ is free from $B_1*B_2$, consider an alternating word $X = a_1 a_2 \dots a_k$ of trace 0 terms with $a_j \in B_1 * B_2$ for $j$ even and $a_j \in A$ for $j$ odd. It suffices to check that $\tau(X) = 0$. Without loss of generality, for any even index $j$ we may assume that $a_j \in B_1 * B_2$ is itself an alternating word of trace 0 letters from $B_1$ and $B_2$ respectively. Similarly, for any odd index $j$, we may assume $a_j$ is a power of either $u_1u_2$ or $u_2^*u_1^*$. After these assumptions, $X$ becomes an alternating word of trace 0 terms of the form $u_1u_2$ or $u_2^*u_1^*$, or trace 0 terms from $B_1$ and $B_2$. 
    
    If it were the case such a product was an alternating word of trace 0 terms from $A_1$ and $A_2$, then by free independence of $A_1$ and $A_2$ we would have $\tau(X) = 0$. However, at present, this might not be the case. More precisely, it might be the case that for some $1 \leq i \leq 2$ and some trace 0 term $b_i \in B_i$ the product $X$ contains one of the terms $u_i^* b_i$, $b_i u_i$, or $u^*_i b_i u_i$. However, since $u_i \in A_i$ was assumed to be orthogonal to $B_i$ we have $\tau(u_i^* b_i) = \tau(b_i u_i) = 0$, so that $u_i^* b_i$ and  $b_i u_i$ are themselves trace 0 terms in $A_i$. Similarly, since $A_i$ was abelian we have $u^*_i b_i u_i = b_i$ is also a trace 0 term in $A_i$. With this observation, it is clear that $X$ is still in fact an alternating word of trace 0 terms from $A_1$ and $A_2$, and hence $\tau(X) = 0$, as desired. 
\end{proof}

\begin{cor}
    The radial MASA $A_r \subset L\mathbb{F}_n$ satisfies the weak FC property.
\end{cor}
\begin{proof}
    Let $u_1, u_2, \dots u_n$ be the freely independent Haar unitaries that generate $\mathbb{F}_n$. For $1 \leq i \leq n$, let $A_i = \{u_i \}''$ so that $\mathbb{F}_n = A_1 * \cdots * A_n$. Furthermore, for $1 \leq i \leq n$, call $B_i = \{u_i + u_i^* \}'' \subset A_i$. Note that the inclusion $B_i \subset A_i$ is isomorphic to the inclusion of the subalgebra of even functions on $[-1, 1]$ inside $L^\infty([-1, 1], \mu)$. Using this correspondence, we see that there exist unitaries $u_i \in A_i$ with $E_{B_i}^{A_i}(u_i) = 0$. Namely, we can let $u_i$ correspond to the even function $|x|/x \in L^\infty([-1, 1], \mu)$. By Lemma \ref{weak FC property lemma}, any one of the Haar unitaries $u_i u_j$, for distinct $1 \leq i, j \leq n$, will be freely independent to the radial subalgebra. 
\end{proof}
\begin{remark}
    Note that although any of these abelian subalgebras $B = \{u_iu_j\}'' \subset L(\mathbb{F}_n)$ is freely independent to the radial subalgebra $A_r$ we do not have $L(\mathbb{F}_n) = A_r\vee B$, even when $n=2$. To see this, note that $B \subset \{u_i, u_j\}''$. Since $u_i$ and $u_j$ were chosen to be freely independent unitaries of order 2 and trace 0, we have $\{u_1, u_2\} \cong L(\mathbb{Z}_2 * \mathbb{Z}_2)$ is amenable. Since $B$ is diffuse abelian but not maximal amenable, by \cite{P83} it cannot be FC and hence $A_r \vee B \neq L(\mathbb{F}_n)$.
\end{remark}

Combining the previous corollary with Corollary \ref{non-separable FC Theorem}, we can further prove an analog of the weak FC conjecture for non-separable free products. To this end, we prove a sequence of additional technical lemmas. We begin with a modification of the previous result. 

\begin{lemma}\label{Lemma: second free independent haar unitary lemma}
    Let $A_1$ and $A_2$ be abelian tracial von Neumann algebras, with $A_1$ diffuse, and let $M = A_1 * A_2$ be their free product. Let $B_1 \subset A_1$, $B_2 \subset A_2$ be von Neumann subalgebras. Suppose $u_1 \in A_1$  is a Haar unitary, and $u_2 \in A_2$ is unitary such that $E_{B_2}^{A_2}(u_2) = 0$. Then $u = u_2u_1u_2^*$ is freely independent to $B_1 * B_2$.
\end{lemma}
\begin{proof}
    The proof is similar to that of Lemma \ref{weak FC property lemma}. Let $u = u_2u_1u_2^*$ and let $A = \{u\}''$.  Note that $A$ is a diffuse abelian subalgebra of $M$ with the set of all powers $u_2u_1^iu_2^*$, for $i\neq 0$, forming an orthogonal basis for $L^2A \ominus \complex 1$.

    To see that $A$ is free from $B_1 * B_2$, consider an alternating word $X = a_1 a_2 \dots a_k$ of trace 0 terms with $a_j \in B_1 * B_2$ for $j$ even and $a_j \in A$ for $j$ odd. It suffices to check that $\tau(X) = 0$. Without loss of generality, for any even index $j$ we may assume that $a_j \in B_1 * B_2$ is itself an alternating word of trace 0 letters from $B_1$ and $B_2$ respectively. Similarly, for any odd index $j$, we may assume $a_j$ is a power of $u_2u_1u_2^*$. After these assumptions, $X$ becomes an alternating word of trace 0 terms of the form $u_2u_1^iu_2^*$, or trace 0 terms from $B_1$ and $B_2$. 
    
    If it were the case such a product was an alternating word of trace 0 terms from $A_1$ and $A_2$, then by free independence of $A_1$ and $A_2$ we would have $\tau(X) = 0$. However, just like in the proof of Lemma \ref{weak FC property lemma}, at present, this might not be the case. More precisely, it might be the case that for some trace 0 term $b \in B_2$ the product $X$ contains one of the terms $u_2^* b$, $b u_2$, or $u^*_2 b u_2$. However, since $u_2 \in A_2$ was assumed to be orthogonal to $B_2$ we have $\tau(u_2^* b) = \tau(b u_2) = 0$, so that $u_2^* b$ and  $b u_2$ are themselves trace 0 terms in $A_i$. Similarly, since $A_2$ was abelian we have $u^*_2 b u_2 = b$ is also a trace 0 term in $A_2$. With this observation, it is clear that $X$ is still in fact an alternating word of trace 0 terms from $A_1$ and $A_2$, and hence $\tau(X) = 0$, as desired. 
\end{proof}

\begin{lemma}\label{Existence of orthogonal unitary lemma}
     Let $A$ be an abelian tracial von Neumann algebra, and let $B \subset A$ be a separable von Neumann subalgebra such that $A \not\prec_A B$ in the sense of Popa's intertwining theory from \cite{P06} (for instance, if $A$ is purely non-separable and $B$ is any separable subalgebra). Then there exists a unitary $u \in A$ such that $E_B^A(u) = 0$.
\end{lemma}
\begin{proof}
    Since $A \not\prec_A B$, there is a net of unitaries $(u_n) \subset A$ such that $\| E_B^A(u_n) \|_2 \to 0$. Because $B$ is separable, we may even assume $(u_n)_{n\in \mathbb{N}}$ is a sequence.

    Call $A_0 \subset A$ the separable von Neumann subalgebra of $A$ generated by $B$ and the unitaries $u_n$ for $n \geq 1$. Then we have $A_0 \cong L^\infty(X, \mu)$  and $B\cong L^\infty(Y, \nu)$ for some standard probability spaces $(X, \mu)$ and $(Y, \nu)$ respectively. Let $\pi: X \to Y$ be the map induced by the inclusion $B \subset A_0$. Then we let $\mu = \int_Y \mu_y d\nu(y)$ be the disintegration of $\mu$ with respect to $\pi$.

    With this notation, note that if $f \in A_0$ is viewed as a function $f(x)$ in $L^\infty(X, \mu)$, then $E_B^A(f) \in B$, viewed as a function $g(y)$ in $L^\infty(Y, \nu)$, is $g(y) = \int_{\pi^{-1}(y)} f(x) d\mu_y(x)$. With this in mind, view each of the unitaries $u_n \in A_0$, for $n \geq 1$, as functions $u_n(x) \in L^\infty(X,\mu)$ with $|u_n(x)| = 1$ for almost every $x$. Then the condition $\|E^A_B(u_n)\|_2 \to 0$ implies that for $\nu$-almost every $y$, $\int_{\pi^{-1}(y)} u_n(x) d\mu_y(x)$ tends to 0 as $n\to \infty$. Unpacking this convergence, we see that for $\nu$-almost every $y \in Y$, the space $L^\infty(\pi^{-1}(y), \mu_y)$ corresponding to the fiber over $y$ is either diffuse or every atom $\{x_0\}$ satisfies

    $$
    \mu_y(\{x_0\})
    \leq 
    \lim_{n \to \infty} \int_{\pi^{-1}(y) \setminus \{x_0\}} |u_n(x)| d\mu_y(x)
    =
    \mu_y(\pi^{-1}(y) \setminus \{x_0\})
    $$

    Because of this inequality, we see immediately that for $\nu$-almost every $y \in Y$, there is a function $u_y \in L^\infty(\pi^{-1}(y), \mu_y)$ with $|u_y| = 1$ that satisfies $\int_{\pi^{-1}(y)} u_y(x) d\mu_y(x) = 0$. We can then choose a unitary $u \in A_0$ that corresponds to the function $u(x) \in L^\infty(X, \mu)$ such that for $\nu$-almost every $y \in Y$ we have $u(x) = u_y(x)$ for all $x \in \pi^{-1}(y)$. This unitary $u$ satisfies $E_B^A(u) = 0$, as desired.     
\end{proof}

\begin{thrm}\label{non-separable weak FC theorem}
    For $2 \leq n \leq  \infty$, let $A_1, A_2, \dots, A_n$ be purely non-separable abelian tracial von Neumann algebras and $M = A_1 * \dots * A_n$. Then for any abelian von Neumann subalgebra $\A \subset M$, there is a diffuse abelian von Neumann subalgebra $\B \subset M$ that is freely independent from $\A$.
\end{thrm}

\begin{proof}
    Without loss of generality, we may assume that the abelian subalgebra $\A$ is a MASA. To begin, we decompose $\A$ into purely non-separable and separable components. That is, let $p_{\rm ns} + p_{\rm s} = 1$ be a partition of unity in $\A$ such that $\A p_{\rm ns}$ is purely non-separable and $\A p _{\rm s}$ is separable. 

    We claim that without loss of generality that we may assume $\A p_{\rm ns}$ fully intertwines inside the first generator MASA $A_1$ of $M$. To see this, notice since $\A p_{\rm ns}$ is a purely non-separable MASA of $p_{\rm ns} M p_{\rm ns}$, Corollary 3.6 of \cite{BDIP23} implies that for some set $K$ there exist projections $(p_k)_{k \in K} \subset \A p_{\rm ns}$ and unitaries $(u_k)_{k \in K}\subset M$ such that $\sum_{k \in K} p_k = p_{\rm ns}$ and for every $k \in K$, $u_k \A p_k u_k^* \subset A_{i_k}$ for some $1 \leq i_k \leq n$.  In fact, since $u_k \A p_k u_k^*$ is a MASA in $u_k p_k M p_k u_k^*$, we have for every $k \in K$, $u_k \A p_k u_k^* = A_{i_k} u_k p_k u_k^*$. In other words, $\A p_{\rm ns} = \sum_{k \in K} u_k^* A_{i_k} u_k p_k$. 

    For $1 \leq i \leq n$, let $e_i$ be the supremum of all projection $u_k p_{k}u_k^*$, for $k \in K$, such that $u_kp_ku_k^*$ belongs to $A_i$. Based on the form of $\A p_{\rm ns}$, we will have $\sum_{i=1}^n \tau(e_i) \leq 1$. Thus, we can find projections $f_i \in A_i$, for $1 \leq i \leq n$, with $e_i \leq f_i$ and $\sum_{i=1}^n \tau(f_i) = 1$. Pick unitaries $v_i \in M$, for $1 \leq i \leq n$, such that $\sum_{i = 1}^n v_i f_i v_i^* = 1$. Then consider the abelian subalgebra $P = \sum_{i = 1}^n v_i A_i f_i v_i^*$. By Theorem \ref{thm: free reassemblies are freely complemented}, $P$ will be freely complemented in $M$. In particular, based on the construction in the proof of Theorem \ref{thm: free reassemblies are freely complemented}, we can find purely non-separable abelian subalgebras $P = P_1,  P_2, \dots, P_n$ such that $M = P_1 * \cdots *P_n$. Moreover, by construction, we see that $\A  p_{\rm ns}$ fully intertwines into this first abelian subalgebra $P = P_1$. So indeed, after replacing the original generator MASAs $A_1, A_2, \dots, A_n$ with these new algebras $P_1, P_2, \dots, P_n$ we may assume $\A p_{\rm ns}$ fully intertwines into $A_1$. 

    So assume that $\A p_{\rm ns}$ fully intertwines into $A_1$. Then for some countable set $K$, there exist projections $(p_k)_{k \in K}$ in $A_1$ and unitaries $(u_k)_{k \in K}$ such that $\A p_{\rm ns} = \sum_{k \in K} u_k A_1 p_k u_k^*$. Let $Q$ be the separable subalgebra of $M$ generated by $\A p_{\rm s}$ and the unitaries $(u_k)_{k \in K}$. Then by a repeated application of Lemma 4.3 of \cite{BDIP23}, there exist separable subalgebras $Q_i \subset A_i$, for $1\leq i \leq n$ such that $Q \subset Q_1 * \cdots * Q_n$. In particular it follows that $\A \subset A_1 * Q_2 * \cdots *Q_n$.

    The result now follows from the earlier lemmas. Using Lemma \ref{Existence of orthogonal unitary lemma}, we can find unitaries $v_i \in A_i$, for $2 \leq i \leq n$, such that $E_{Q_i}^{A_i}(v_i) = 0$. If $n \geq 3$, it follows from Lemma \ref{weak FC property lemma} that $\B = \{v_2 v_3\}''$ is a diffuse abelian subalgebra freely independent to $\A$. Conversely, if $n = 2$, then take any Haar unitary $v_1 \in A_1$, which exists because $A_1$ is purely non-separable. Then by Lemma \ref{Lemma: second free independent haar unitary lemma}, the diffuse abelian subalgebra $\B = \{v_2v_1 v_2^*\}''$ will be freely independent to $\A$.
\end{proof}

In the proof of Theorem \ref{non-separable weak FC theorem}, notice that the assumption that the subalgebra $\A \subset M = A_1 * \cdots * A_n$ is abelian was only used to understand the structure of the purely non-separable part of $\A$. In fact, in the case $N \subset M$ is any separable algebra the same argument can be repeated. 

\begin{thrm}\label{Theorem: Separable Algebras have weak FC}
    For $2 \leq n \leq  \infty$, let $A_1, A_2, \dots, A_n$ be purely non-separable abelian tracial von Neumann algebras and $M = A_1 * \dots * A_n$. Then for any separable von Neumann subalgebra $\N \subset M$, there is a diffuse abelian von Neumann subalgebra $\B \subset M$ that is freely independent from $\N$.
\end{thrm}
\begin{proof}
    By repeated application of Lemma 3.4 of \cite{BDIP23}, there exists separable subalgebras $Q_i \subset A_i$, for $1\leq i \leq n$, such that $\N \subset Q_1 * \cdots *Q_n$. Applying Lemma \ref{Existence of orthogonal unitary lemma}, we can find unitaries $v_i \in N_i$, for $1 \leq i \leq n$ such that $E^{A_i}_{Q_i}(v_i) = 0$. Lemma \ref{weak FC property lemma} then implies there exists a diffuse abelian von Neumann subalgebra $\B \subset M$ freely independent to $\N$.
\end{proof}

\begin{remark}
    Consider a separable subalgebra $N \subset M  = A_1 * \cdots *A_n$ as in the statement of Theorem \ref{Theorem: Separable Algebras have weak FC}. Then although we can say there exists an abelian subalgebra $B \subset M$ that is freely independent to $N$, it cannot be the case that $M = N \vee B$. More precisely, a separable subalgebra $N \subset M$ cannot be freely complemented by an abelian algebra. To see this, suppose that $M = N * B$ for some (potentially non-separable) tracial von Neumann algebra $B$. Then we can compute the sans-rank of $B$, (as defined in Section 2 of \cite{BDIP23}) to be $r_{\rm ns}(B) = r_{\rm ns}(M) - r_{\rm ns}(N) = n$. In particular, since $r_{\rm ns}(B) > 1$ it must be that $B$ is not abelian. In fact, since $r_{\rm ns}(B) = r_{\rm ns}(M)$ it must be that every purely non-separable singular MASA of $M$ fully intertwines into $B$. In particular, each of the generator MASAs $A_i$, for $1 \leq i \leq n$, of $M$, each of which are maximal amenable by \cite{P83}, fully intertwine into $B$. So in fact, we even obtain that if a separable subalgebra $N \subset M$ is freely complemented by some $B$, then $B$ be must be non-amenable.
\end{remark}


\end{document}